\documentclass[12pt]{amsart}
\usepackage{mathrsfs}
\usepackage{latexsym,amssymb,amsmath,amsfonts,graphicx,epstopdf,subfigure,float}
\usepackage{bbm}
\usepackage{pifont}
\usepackage{cite}
\usepackage{multirow}
\usepackage{appendix}

\def\diag{\mathop{\rm diag}\nolimits}

\newtheorem{thm}{Theorem}[section]
\newtheorem{lem}{Lemma}[section]
\newtheorem{prop}{Proposition}[section]
\newtheorem{coro}{Corollary}[section]

\newtheorem{defn}[lem]{Definition}

\numberwithin{equation}{section}
\newtheorem{remark}{Remark}[section]

\def\R{\mathbb R}

\def\Z{\mathbb Z}

\def\mathscr{\mathcal }

\def\diag{\text{diag}}

\newcommand{\btau}{{\boldsymbol{\tau}}}

\newcommand{\bomega}{{\boldsymbol{\omega}}}

\newcommand{\bsigma}{{\boldsymbol{\sigma}}}

\newcommand{\ba}{{\mathbf{a}}}

\newcommand{\bd}{{\mathbf{d}}}

\newcommand{\bi}{{\mathbf{i}}}

\newcommand{\bp}{{\mathbf{p}}}

\newcommand{\bD}{{\mathcal D}}

\newcommand{\SD}{{\mathcal D}}

\newcommand{\SH}{{\mathcal H}}

\newcommand{\SV}{{\mathcal V}}
 \usepackage{color}
\begin{document}
\title{Distribution of  $\delta$-connected components of  self-affine sponge of Lalley-Gatzouras  type }

\author{Yan-fang Zhang}
\address{School of Science, Huzhou University, Huzhou, 313000, China;}
\email{03002@zjhu.edu.cn}

\author{yong-qiang yang $^*$ } \address{Department of Mathematics and Statistics, Central China Normal University, Wuhan, 430079, China}
\email{yangyq@mails.ccnu.edu.cn}

\date{\today}
\thanks {The work is supported by the start-up research fund from Huzhou University No. RK21089.}


\thanks{* The correspondence author.}

\maketitle

\begin{abstract} Let $(E, \rho)$ be a metric space and let $h_E\left( \delta \right)$ be the cardinality of the set of $\delta$-connected components of $E$.
In literature, in case of that   $E$ is a self-conformal set satisfying the open set condition or $E$ is a self-affine Sierpi\'nski sponge,  necessary and sufficient condition is given for the validity of the relation
$$
h_E(\delta)\asymp \delta^{-\dim_B E}, \text{ when }\delta\to 0.
$$
In this paper, we generalize the above result to   self-affine sponges of Lalley-Gatzouras type;
actually   in this case, we show that there exists a Bernoulli measure $\mu$ such that
 for any cylinder $R$, it holds that
$$
h_R(\delta)\asymp \mu(R) \delta^{-\dim_B E}, \text{ when }\delta\to 0.
$$
\end{abstract}

\section{Introduction}

Let $\left( E,\rho \right)$ be a metric space and let $\delta >0$. Two points $x,y\in E$ are said to be
$\delta$-equivalent if there exists a sequence $\left\{ x_1=x,x_2,...,x_{k-1},x_k=y \right\} \subset E$ such that $\rho \left( x_i,x_{i+1} \right) \le \delta$ for $1\le i\le k-1$. A $\delta$-equivalent class of $E$ is called a \emph{$\delta$-connected component} of $E$. We denote by $h_E\left( \delta \right)$ the cardinality of the set of $\delta$-connected components of $E$.

A notion closely related to $\delta$-connected component is the gap sequence of a compact metric space, which has been studied by
   many mathematicians; see for instance, \cite{BT54,Fal95,LP93, RRY08, DWX15}.
Some early works  (\cite{LP93, Fal95, DWX15}) observed that for some totally disconnected
self-similar sets and self-conformal sets $E$,
the gap sequence, which we denote by $\{g_E(n)\}_{n\geq 1}$, is comparable to $n^{-1/\dim_B E}$,
which is written as
$$
g_E(n)\asymp n^{-1/\dim_B E}, \ n\geq 1,
$$
where $\dim_B$ denotes the box dimension.
(Two functions $f,g:X\to {\mathbb R}$ are said to be \emph{comparable}, denoted by $f\asymp g$, if there exists a constant $c>0$ such that $c^{-1}g(x)\le f(x) \le cg(x)$   for all $x\in X$.)

Let $(E,\rho)$ be a compact metric space and let $\gamma>0$.
 It is shown (Miao, Xi and Xiong \cite{MXX17}, Zhang and Huang \cite{ZH22})   that
 $$g_E(n)\asymp n^{-1/\gamma} \Leftrightarrow h_E(\delta)\asymp\delta^{-\gamma}.$$
 Motivated by this relation, Zhang and Huang \cite{ZH22} proposed the following definition.

 \begin{defn}\emph{Let $E$ be a compact metric space.
 We say $E$ satisfies the \emph{power law with index $\gamma$} if $h_E\left( \delta \right) \asymp \delta ^{-\gamma}$; if $\gamma =\dim_BE$, in addition, we say $E$ satisfies the \emph{maximal power law}.}
 \end{defn}


There are many works devoted to the maximal power law of attractors of IFS
(\cite{LP93, Fal95, DWX15,MXX17, LMR20, LR18, HZ22, ZH22}).
An \emph{iterated function system} (IFS) is a  family of contractions $\Phi=\{\varphi_j\}_{j=1}^N$  on a metric space $X$. In this paper, we will always assume that
   all $\varphi_j$ are injections. The \emph{attractor} of the IFS is the unique nonempty compact set $E$ satisfying
$E=\bigcup_{j=1}^N\varphi_j(E)$; especially, it is called a \emph{self-similar set} if all $\varphi_j$ are similitudes.
An IFS $\{\varphi_j\}_{j=1}^N$ is said to satisfy the \emph{open set condition} (OSC), if there is a bounded nonempty open set $U\subset\mathbb{R}^{d}$ such that for all $1\le i\le N$, $\varphi_i(U)\subset U$ and $\varphi_i(U)\cap\varphi_j(U)=\emptyset$ for $1\le i\ne j\le N$;
moreover, if $U\cap E\neq \emptyset$, then we say $E$ satisfies the \emph{strong open set condition} (SOSC) and we call $U$ a \emph{feasible strong open set} for $E$.

Lapidus, Pomerance   \cite{LP93} and Falconer \cite{Fal95}  confirmed  maximal power law for self-similar sets $E\subset \R$ satisfying the  strong separation  condition, and
Deng, Wang and Xi \cite{DWX15} generalized this result to the self-conformal sets.
Recently,  Huang and Zhang \cite{HZ22} gave a complete answer
in case of the
self-conformal set satisfying the open set condition.

A point $x\in E$ is called a trivial point of $E$ if $ \{ x \}$ is a connected component of $E$.

\begin{prop}[\cite{HZ22}]\label{prop:conformal}  Let $K$ be a $\mathcal{C}^{1+\alpha}$ self-conformal set generated by the IFS $\{S_j\}_{j=1}^N$.
Assume $K$ satisfies the OSC.
Then the following statements are equivalent:

 \emph{(i)} $K$ satisfies the maximal power law.

\emph{(ii)} There exists a strong open set $U$ such that $U$ contains  trivial points  of $K$.

\emph{(iii)} Every strong open set  for $K$ contains trivial points of $K$.

\emph{(iv)} The set of trivial points of $K$ is dense in $K$.
\end{prop}

There are several works devoted to the maximal power law of Bedford-McMullen carpets\cite{MXX17, LMR20, LR18} and self-affine Sierpi\'nski sponge \cite{ZH22}.
 Let $d\geq 2$ and let
$$ 2\leq n_1\leq  n_2< \dots\leq n_d$$
 be a sequence of  integers. Let $\Gamma=\diag(n_1,\dots, n_d)$ be the $d\times d$ diagonal matrix. Let $\bD=\{\bi_1,\dots,\bi_N\}\subset\prod_{j=1}^d\{0,1,\dots,n_j-1\}$.
  For  $\bi\in\bD$ and $z\in\mathbb{R}^d$, we define $S_{\bi}(z)=\Gamma^{-1}(z+\bi)$.
     The attractor of the IFS $\{S_{\bi}\}_{\bi\in\bD}$ , which is denoted by $E=K(\Gamma,\bD)$,
is called a $d$-dimensional \emph{generalized self-affine Sierpi$\acute{\text{n}}$ski sponge} (see Olsen \cite{Olsen07}). The set $E$ is called a \emph{fractal cube} if $n_1=n_2=\cdots=n_d$;
  it is called a \emph{self-affine Sierpi$\acute{\text{n}}$ski sponge} if
 $n_1<n_2<\cdots<n_d$ (see Kenyon and Peres \cite{KP96}), and
especially when $d=2$, $E$ is called a \emph{Bedford-McMullen carpet}.

We say $E$ is \emph{non-degenerated} if $E$ is not contained in a $(d-1)$-dimensional face
of the cube $[0,1]^d$.
For $1\leq j\leq d$, let $\pi_j:\R^d\to \R^j$ be the projection
\begin{equation}\label{eq:pi-j}
\pi_j(x_1,\dots, x_d)=(x_1,\dots, x_j).
\end{equation}

\begin{prop}[\cite{ZH22}] \label{prop:law} Let $E$ be a non-degenerated self-affine Sierpi$\acute{\text{n}}$ski sponge. Then $E$
satisfies the maximal power law if and only if  $E$ and all $\pi_j(E)(1\le j\le d-1)$ possess trivial points.
\end{prop}

 The main purpose of the present paper is to prove a stronger version of
Proposition \ref{prop:law} for   Lalley-Gatzouras sponges.
A diagonal self-affine sponge $\Lambda$ is said to be a  \emph{Lalley-Gatzouras sponge} if the generating IFS satisfies a coordinate ordering condition as well as a
neat projection condition (\cite{Lalley92, Das16}).
See Section 3 for the precise definitions.

\begin{remark}\emph{
The diagonal self-affine sponge is an important class of self-affine sets received a lot
of studies in recent years,
Baranski \cite{Baranski07}, Das and Simmons \cite{Das16},
Feng and Wang \cite{FW05}, Mackay \cite{MM11},
Peres \cite{Peres01}
  Banaji and Kolossvary \cite{BK21} on dimension theory,
King \cite{King95}, Jordan and Rams \cite{JR11}, Barrel and Mensi\cite{Bar07},
Olsen \cite{Olsen07}, Reeve\cite{Reeve10} on multifractal analysis,
Li \text{et al } \cite{LLM13}, Rao \textit{et al} \cite{Rao19},
Yang and Zhang \cite{YZ20}, on Lipschitz classification, etc.
The distribution of $\delta$-connected components
 illustrates the metric and topology properties of the self-affine sponges from a new point
of view.
}
\end{remark}

However, this generalization is not trivial. The difficulty comes from the fact that
given a cylinder $E_I$ of a Lalley-Gatzouras type sponge $E$,   the relation between   $h_E(\delta)$ and $h_{E_I}(\delta)$ is unclear.
In this paper, we overcome this difficulty by using
a recent result of  Huang \text{et al} \cite{HRWX} on box-counting measures.

Denote $\Sigma=\{1,\dots, N\}$ and set $\Sigma^*=\bigcup_{n\geq 0} \Sigma^n$.
For $I=i_1\dots i_n\in \Sigma^*$, we call
$E_I=\varphi_I(E)=\varphi_{i_1}\circ\dots\circ\varphi_{i_n}(E)$
an $n$-th \emph{cylinder} of $E$. We call $\delta(z+[0,1]^d)$ a \emph{$\delta$-mesh} box if $z\in \Z^d$.
For $A\subset \R^d$, we use
\begin{equation}\label{eq_N}
N_A(\delta)
\end{equation}
to denote the number of $\delta$-mesh boxes
intersecting $A$. It is well known that
$$\dim_B A=\lim_{\delta\to 0}-\frac{\log N_A(\delta)}{\log \delta}$$
 if the limit exists
(see for instance, \cite{Falconer90}).

For a Lalley-Gatzouras sponge, a special Bernoulli measure, which we will
call the \emph{canonical Bernoulli measure} (see  Section 3 for precise definition), is closely related to the box-dimension ( \cite{Lalley92,HRWX}). Recently, Huang \textit{et al} proved the following  result.

 \begin{prop}\label{prop:sponge} (\cite{HRWX}) Let $\Lambda$ be a
    Lalley-Gatzouras sponge.    Then
 the canonical Bernoulli measure $\mu$  is a \textbf{cylinder box-counting measure} in the sense that,
 there is a constant $M_0>0$ such that for any cylinder $R$ of $\Lambda$, and any $\delta<S(R)$,
 the shortest side of $R$,
 it holds that
 \begin{equation}
M_0^{-1}\mu(R)\delta^{-\dim_B \Lambda}\leq  N_R(\delta)\leq M_0\mu(R)\delta^{-\dim_B \Lambda}.
 \end{equation}
\end{prop}

Motivated by the above result,  we  propose the following concept.

\begin{defn}[Component-counting measure]
 \emph{ Let $\Lambda$ be the attractor of an IFS $\Phi$,  and let $\mu$ be
  a  finite Borel measure on $\Lambda$.
  If there is a constant $M\geq 1$ such that for any  cylinder $R$ of $\Lambda$, there exists
  a $\delta_0=\delta_0(R)$ such that
$$
M^{-1} \mu \left( R \right) \delta ^{-\dim_B \Lambda} \leq h_R\left( \delta \right) \leq M \mu \left( R \right) \delta ^{-\dim_B \Lambda},
$$
for $\delta <\delta_0$,
then we call  $\mu$ a \emph{component-counting measure} of $\Lambda$  .
}
\end{defn}

\begin{remark}\label{rem:homo}\emph{Clearly, if $\Lambda$ admits a component-counting measure, then not only $\Lambda$, but all its
cylinders satisfies the maximal power law with a uniform constant $M$.
This reflects a kind of homogenity
of $\Lambda$.}
\end{remark}

First, we add one more equivalent statement to the list in Proposition \ref{prop:conformal}.

\begin{thm}\label{thm:conformal} Let $K$ be a ${\mathcal C}^{1+\alpha}$ self-conformal set satisfying the OSC,  and let $\mu$
be the Gibbs measure. Then   $K$ satisfies the maximal power law if and only if
 $\mu$ is a component-counting measure.
\end{thm}

 A   Lalley-Gatzouras  sponge $\Lambda$    is said to be \emph{degenerated} if $\Lambda$ is contained in a face of $\left[ 0,1 \right] ^d$ with dimension $d-1$.
 From now on, we will always assume that $\Lambda$ is a non-dengerated Lalley-Gatzouras sponge.
Under this assumption,  $\Lambda$ satisfies the strong open set condition with the open set $V=(0,1)^d$. Denote $\Lambda _j=\pi _j\left( \Lambda \right)$
where $\pi_j$ is defined by \eqref{eq:pi-j}; clearly
  $\Lambda_j$  is also non-degenerated, and $V_j=(0,1)^j$ is a feasible strong open set
for $\Lambda_j$. (See Lemma \ref{lem:island}.)
A trivial point $x$ of $\Lambda$ is called an
\emph{inner trivial point} of $\Lambda$ if $x\in (0,1)^d$.

The main result of the present paper is the following two theorems.

\begin{thm}\label{thm:affine} Let $\Lambda$ be a non-degenerated
  Lalley-Gatzouras type sponge, and  let $\mu$ be the canonical Bernoulli measure.
 Then the following statements are equivalent

(i)   $\mu$ is a component-counting measure.

(ii)  $\Lambda$ satisfies the maximal power law.

(iii) $V_j=(0,1)^j$ possesses  trivial points of
$\Lambda_j=\pi_j(\Lambda)$ for every $1\leq j\leq d$.

(iv)  The trivial points of $\Lambda_j$ is dense in
$\Lambda_j$ for every $1\leq j\leq d$.
\end{thm}

\begin{remark} \emph{
Zhang and Xu \cite{ZX22} proved that if $\Lambda$ is a slicing self-affine sponge (see Remark \ref{rem:slicing} for a definition), then $\Lambda$ possesses inner trivial points   as soon as    it possesses trivial points.
This explains why in Proposition \ref{prop:law}, we  require that $\Lambda_j$ possesses
trivial points instead of  inner trivial points of $\Lambda_j$.
}
\end{remark}

\begin{thm}\label{thm:dropp}  Let $\Lambda$ be a non-degenerated
  Lalley-Gatzouras sponge, and  let $\mu$ be the canonical Bernoulli measure. If $\Lambda$ does not satisfy the maximal power law, then there exists
  a real number $0<\chi <\dim_B \Lambda$ such that for any cylinder $R$,
\begin{equation}\label{eq:drop}
h_R\left( \delta \right) =O\left( \mu \left( R \right) \delta ^{-\chi} \right) ,\ \ \delta \rightarrow 0.
\end{equation}
\end{thm}

\begin{remark} \emph{ We remark that if a cylinder box-counting measure $\mu$ is also a component-counting measure,   then for any cylinder $R$
of $\Lambda$ and   $\delta$ small enough, it holds that
$$
h_R(\delta)\asymp N_R(\delta),
$$
which means that even locally,  a large portion of $\delta$-connected components are very `small'.
}
\end{remark}

The paper is organized as follows. Theorem \ref{thm:conformal} is proved in Section 2.
In Section 3, we recall some known results about Lalley-Gatzouras type sponges.
In Section 4, we give some notations and lemmas.
In Section 5, we deal with the case that there exists $j$ such that $\Lambda_j$ contains
no inner trivial points.
Theorem \ref{thm:dropp}   and Theorem \ref{thm:affine} are proved in Section 6.

\section{\textbf{Component-counting measures of self-conformal sets}}

Let $\alpha>0$. A conformal map $S: V\to \R^d$ is $\mathcal{C}^{1+\alpha}$ differentiable if there exists a constant $C>0$ such that
\begin{equation}\label{Holder}
\left ||S'(x)|-|S'(y)|\right |\le C|x-y|^\alpha \text{ for all } x, y\in V.
\end{equation}
The attractor of an IFS $\{S_j\}_{j=1}^N$ is called a  $\mathcal{C}^{1+\alpha}$ \emph{self-conformal set},
 if all maps in the IFS are $\mathcal{C}^{1+\alpha}$ differentiable. 

In this section, we always assume that  $K$ is a $\mathcal{C}^{1+\alpha}$ self-conformal set generated by the IFS $\{S_j\}_{j=1}^N$. Let $s=\dim_H K(=\dim_B K)$
and $\mu$ be the Gibbs measure on $K$.

For $A\subset\mathbb{R}^d$, we denote $|A|$ the diameter of $A$.
The following lemmas are well-known.

\begin{lem}[Principle of bounded distortion \cite{Patzschke97,Peres01}]\label{lem:distortion}
Let $E$ be a compact subset of $V$. Then there exists  $C_1\geq 1$ such that for any $\bsigma\in\Sigma^*$,
$$
C_1^{-1}|S_\bsigma(E)|\cdot|x-y|\le |S_\bsigma(x)-S_\bsigma(y)|\le C_1|S_\bsigma(E)|\cdot|x-y|
$$
for all $x,y\in E$.
\end{lem}

Strings $\bomega,\btau\in\Sigma^*$ are \emph{incomparable} if each is not a prefix of the other.

\begin{lem}[Alfors regularity \cite{Patzschke97,Peres01}] \label{lem:Alfors}  If $K$ satisfies the OSC, then  $\mu\asymp\SH^s|_K$, and there is a constant $C_2>0$ such that for any $\bsigma\in \Sigma^*$,
$$
C_2^{-1}|K_\bsigma|^s\leq \mu(K_\bsigma)\leq C_2 |K_\bsigma|^s.
$$
 Moreover, $K_\bomega$ and $K_\btau$ are disjoint in $\mu$ provided that
$\bomega$ and $\btau$ are incomparable.
\end{lem}

\begin{proof}[\textbf{Proof of  Theorem \ref{thm:conformal}}] 
First, let $K$ be a self-conformal set satisfying the maximal power law.
 Let $K_I$ be a cylinder and let $\delta<|K_I|$.
Set $\delta'=C_1\delta/|K_I|$, where $C_1$ is the constant in Lemma \ref{lem:distortion}.
Let $U_1$ and $U_2$ be two $\delta'$-connected components of $K$.
Then by Lemma \ref{lem:distortion}, $x\in S_I(U_1)$ and $y\in S_I(U_2)$ belongs to different
$\delta$-connected component of $K_I$. Therefore,
$$
\begin{array}{rlr}
h_\delta(K_I)& \geq h_{\delta'}(K) &\text{(By the similarity mentioned above.)}\\
             &\geq M (\delta')^{-s}   &\text{(By the maximal power law.)}\\
             &=MC_1^{-s}|K_I|^s\delta^{-s} &\\
             &\geq M(C_1C_2)^{-s}\mu(K_I) \delta^{-s}.  &\text{(By Lemma \ref{lem:Alfors}.)}
\end{array}
$$
The other direction inequality can be proved in the same manner. It follows that
$\mu$ is a component-counting measure.

On the other hand, if $\mu$ is a component-counting measure, then obviously $K$ satisfies
the maximal power law (see Remark \ref{rem:homo}). The theorem is proved.
\end{proof}

\section{\textbf{Self-affine sponges of Lalley-Gatzouras type}}

In this section, we recall some known results on Lalley-Gatzouras type sponges.

 \subsection{Lalley-Gatzouras type sponges}  We call $f:\R^d\to \R^d$, $f(x)=Tx+b$ a \emph{diagonal self-affine mapping}
if $T$ is a $d\times d$ diagonal matrix such that all the diagonal entries are positive numbers.
An IFS $\Phi=\{\phi_j(x)\}_{j=1}^m$ is called a \emph{diagonal self-affine IFS}
if all the maps $\phi_j(x)$ are distinct diagonal self-affine contractions; the attractor is called
a \emph{diagonal self-affine sponge}, and we denote it by $\Lambda_\Phi$. Without loss of generality, we will always assume that  $\Lambda_\Phi\subset [0,1]^d$.
The following is an alternative definition.

\begin{defn}[Diagonal IFS \cite{Das16}] \emph{
 Let $d\geq 1$ be an integer.
 For each $i\in \left\{ 1,...,d \right\}$, let $A_i=\left\{ 0,1,...,n_i-1 \right\}$ with $n_i\ge 2$, and let $\varPhi _i=\left( \phi _{a,i} \right) _{a\in A_i}$ be a collection of contracting similarities of $\left[ 0,1 \right]$, called the base IFS in coordinate $i$. Let $A=\prod_{i=1}^d{A_i}$, and for each $a=\left( a_1,\cdots ,a_d \right) \in A$, consider the contracting affine maps $\phi _a:\left[ 0,1 \right] ^d\rightarrow \left[ 0,1 \right] ^d$ defined by the formula
$$
\phi _a\left( x_1,\cdots x_d \right) =\left( \phi _{a,1}\left( x_1 \right) ,\cdots ,\phi _{a,d}\left( x_d \right) \right)
$$
where $\phi _{a,i}$ is shorthand for $\phi _{a_i,i}$ in the formula above. Then we can get
$$
\phi _a\left( \left[ 0,1 \right] ^d \right) =\prod_{i=1}^d{\phi _{a,i}\left( \left[ 0,1 \right] \right)}\in \left[ 0,1 \right] ^d
$$
}

\emph{Given $\mathcal{D}\subset A$, we call the collection $\varPhi =\left( \phi _a \right) _{a\in \mathcal{D}}$ a \emph{diagonal IFS}, and we call its invariant set $\Lambda$ a \emph{diagonal self-affine sponge.}
}
\end{defn}

\begin{remark}\label{rem:slicing} \emph{A diagonal self-affine IFS $\varPhi$ is called a slicing self-affine IFS, if for each $i\in \left\{ 1,...,d \right\}$,
$$
\left[ 0,1 \right] =\phi _{0,i}\left[ 0,1 \right) \cup \cdots \cup \phi _{n_i-1,i}\left[ 0,1 \right)
$$
is a partition of $[0,1)$ from left to right; in this case, we call $\Lambda$ a
\emph{slicing self-affine sponge}.
In particular, if $d=2$, then we call $\Lambda$ a \emph{Bara\'{n}ski carpet}.
Furthermore, if for each $1\leq j\leq d$,  the contraction ratios of $\phi _{a,j}$, $a\in \SD$, are all equal, then $\Lambda$ is the \emph{self-affine Sierpi$\acute{\text{n}}$ski sponge}.
}
\end{remark}

We say that $\Phi$ satisfies the \emph{coordinate ordering condition} if
\begin{equation}\label{eq:order}
 \phi '_{a,1}  >\cdots >  \phi '_{a,d}, \quad \text{ for all } a\in \mathcal{D},
\end{equation}
where $f'$ denotes the derivative of the function $f$.

 Recall that $\pi_j(x_1,\dots, x_d)=(x_1,\dots, x_j)$. Let
$$
\varPhi _{\left\{ 1,\cdots ,j \right\}}=\left( \left( \phi _{a,1},\cdots ,\phi _{a,j} \right) \right) _{a\in \pi _j\left( \mathcal{D} \right)},
$$
which is an IFS on $\R^j$. Clearly $\Lambda _j=\pi _j\left( \Lambda \right)$  is the attractor of the IFS $\varPhi _{\left\{ 1,\cdots ,j \right\}}$.

\begin{defn} [\cite{Das16}] \label{def:good}\emph{
Let $\Lambda_\Phi$ be a diagonal self-affine sponge satisfying \eqref{eq:order}.
 We say $\Phi$ satisfies the \emph{neat projection condition},
 if for each $j\in \{1,\dots, d\}$, the IFS $\Phi_{\{1,\dots,j\}}$ satisfies
 the OSC with the open set $\mathbb{I}^j=(0,1)^j$,  that is,
$$
\left\{\phi_{\bd,\{1,\dots, j\}}(\mathbb{I}^j)\right\}_{\bd\in\pi_j(\SD)}
$$
are disjoint.} \emph{Moreover, we say a diagonal self-affine spong $\Lambda_\Phi$ is
of \emph{Lalley-Gatzouras type} if it  satisfies the coordinate ordering condition as well as the
neat projection condition.}
\end{defn}

\subsection{The canonical Bernoulli measure}
Let $\bp=(p_\bd)_{\bd\in \SD}$ be a probability weight. The unique probability measure $\mu_\bp$
satisfying
$$
\mu_\bp(\cdot)=\sum_{\bd\in \SD} p_\bd \phi_\bd^{-1}(\cdot)
$$
is called the \emph{Bernoulli measure} determined by the weight $\bp$.

Let $\Lambda_{\Phi}$ be a   Lalley-Gatzouras type sponge.
Now we define a sequence $\{\beta_j\}_{j=1}^d$
related to $\Lambda_\Phi$.
Let $\beta_1>0$ be the unique real number satisfying
$$\sum\limits_{f_1\in \Phi_{\{1\}}} (f_1')^{\beta_1}=1.$$
If $\beta_1,\dots, \beta_{j-1}$ are defined, we define $\beta_j>0$ to be the unique real number such that
\begin{equation}\label{eq:flag}
\sum\limits_{(f_1,\dots, f_j)\in \Phi_{\{1,\dots, j\}}} \prod_{k=1}^j (f_k')^{\beta_k}=1.
\end{equation}

Next, for ${\mathbf f}=(f_1,\dots, f_j)\in \Phi_{\{1,\dots, j\}}$, define
\begin{equation}\label{eq:weight}
p_{\mathbf f}=\prod_{k=1}^j (f_k')^{\beta_k}.
\end{equation}
Let $\mu_j$ be the Bernoulli measure on $\Lambda_j$ defined by the weight
$(p_{\mathbf f})_{{\mathbf f} \in \Phi_{\{1,\dots, j\}}}$.
Especially, we denote $\mu:=\mu_d$, and we call $\mu$ the \emph{canonical Bernoulli measure}
of $\Lambda$. It is shown that

\begin{thm} (\cite{HRWX}) Let $\Lambda$ be a Lalley-Gatzouras type sponge, then
$$\dim_B \Lambda_{\Phi}=\sum_{j=1}^d \beta_j,$$
 and the canonical Bernoulli measure $\mu$  is a cylinder box-counting measure.
 \end{thm}

 Denote
 \begin{equation}
 \alpha _j:=\dim_B \Lambda_j=  \sum_{k=1}^j{\beta _k}.
 \end{equation}
  Especially $\dim_B \Lambda=\alpha_d:=\alpha$.

\begin{remark}\emph{ Let $R$ and $\widetilde{R}$ be two $\ell$-th cylinders of $\Lambda_\Phi$, and $\nu$ be
a Bernoulli measure of $\Lambda_\Phi$. Then $\nu(R\cap \widetilde{R})=0$ is always true.
See for instance \cite{Rao19}.
}
\end{remark}

%


 \begin{lem}\label{lem:deg}   Let $\Lambda$ be a non-degenerated  diagonal self-affine sponge of Lalley-Gatzouras type.  Then all $\Lambda_j$, $1\leq j\leq d$, are non-degenerated.
 \end{lem}

 \begin{proof}  Define $\pi'_k(x_1,\dots, x_d)=x_k$.
 Then $\Lambda$ is non-degenerated if and only if that
 for each $1\leq k\leq d$, $\pi'_k(\Lambda)\neq \{0\}$ and $\pi'_k(\Lambda)\neq \{1\}$.
 Since for $k\leq j$, $\pi'_k\circ \pi_j=\pi'_k$, we obtain the lemma.
 \end{proof}

\section{\textbf{Notations and lemmas}}

In this section, we always assume that
  $\Lambda$  is a non-degenerated generalized Lalley-Gatzouras type sponge.
  We will use $\SD$ as the alphabet and denote $\SD^*=\bigcup_{n\geq 1}\SD^n$.
  For $i_1\dots i_k\in \SD^k$, we call $\Lambda_{i_1\dots i_k}:=\phi_{i_1\dots i_k}(\Lambda)$
  a $k$-th cylinder; moreover,  we set
 \begin{equation}
 S(W)=\prod_{j=1}^k\phi_{i_j,d}'
 \end{equation}
to be the `shortest side' of $W$. Let
\begin{equation}
r_*=\min \left\{ \phi '_{a,d};\ a\in \mathcal{D} \right\}.
\end{equation}
Let
 \begin{equation}
 \mathcal{V}_{\delta}=
 \{\Lambda_{i_1\dots i_k};~S(\Lambda_{i_1\dots i_k})<\delta/r_*
 \text{ and } S(\Lambda_{i_1\dots i_{k-1}})\geq \delta/r_*\}
 \end{equation}
and we call it the \emph{$\delta$-blocking} of $\Lambda$.
 Clearly for $H\in {\mathcal V}_\delta$, it holds that
 $$ \delta\leq S(H) < \delta/r_*.$$

The following lemma is obvious.

\begin{lem}\label{lem:box}

(i) For any bounded set  $A\subset \R^d$,
it holds that $h_A(\delta)\leq 3^dN_A(\delta)$.

(ii) $h_A(\delta)+h_B(\delta)\geq h_{A\cup B}(\delta)$.

(iii) $h_A(\delta)\geq h_{f(A)}(\delta)$ if $f$ is a contractive map.
\end{lem}

 We use $\partial E$ to denote the boundary of a set $E\subset \R^d$.

Let $W=\Lambda_I$ be a cylinder of $\Lambda$. We use
$$N_W^b(\delta)$$
to denote the number of $\delta$-mesh boxes intesecting  $\phi_I(\partial [0,1]^d\cap \Lambda)$.
The main goal of this section is to  estimate $N_W^b(\delta)$.
To this end, our strategy is to estimate the $\mu$-measure of the set
\begin{equation}\label{eq:mb}
\Lambda_m^b=
  \cup\{\Lambda_I; ~ I\in \SD^m \text{ and } \varphi_{I}([0,1]^d)\cap \partial [0,1]^d\neq \emptyset\},
\end{equation}
and then use the fact that $\mu$ is a box-counting measure.

Let $F_1,\dots, F_{2d}$ be the $(d-1)$-faces of the $[0,1]^d$.
For $1\leq j\leq 2d$, denote
$$
\mathcal{D}^{\left( j \right)}=\left\{a\in \SD; ~\phi_{a}\left(  [ 0,1  ] ^d \right) \cap F_j\ne \emptyset\right\}.
$$
Denote
$$
Q_j=\sum_{a\in \SD^{(j)}} \mu(\Lambda_a)
$$
and set
\begin{equation}\label{eq:Q}
Q=\max\{Q_j;~j=1,\dots, 2d\}.
\end{equation}
Since $\Lambda$ is non-degenerated, we conclude that for each $j$, $\SD^{(j)}$ is a proper subset of
$\SD$ and hence $Q_j<1$.

\begin{lem}\label{lem:QQ}  For each integer $m\geq 1$, we have
$$
\mu(\Lambda_m^b)= \sum_{j=1}^{2d} Q_j^m\leq 2dQ^m.
$$
\end{lem}

\begin{proof} Clearly $\varphi_I([0,1]^d)\cap \partial F_j \neq \emptyset$
if and only if $I\in (\SD^{(j)})^m$, so the measure of the union of such cylinders
is $Q_j^m$. The lemma is proved.
\end{proof}

\begin{lem}\label{lem:boundary-box}
Let $W$ be a cylinder of $\Lambda$ and let $m\geq 1$ be an integer. If $\delta<r_*^mS(W)$, then
$$
N_W^b(\delta)\leq   2dM_0Q^m\mu(W)\delta^{-\alpha},
$$
where $M_0$ is the constant in Proposition \ref{prop:sponge}.
\end{lem}
\begin{proof}Let $W=\Lambda_I$  be a cylinder.  Since $\delta<r_*^mS(W)$,
if a $\delta$-mesh box $B$ intersects $\partial W$, then $B\subset \varphi_I(\Lambda_m^b)$.
By the definition of $\mu$, we have $\mu(\varphi_I(A))=\mu(W)\mu(A)$ for any  cylinder $A$ of $\Lambda$.
Hence, by Lemma \ref{lem:QQ},
 $$\mu(\varphi_I(\Lambda_m^b))\leq 2dQ^m\mu(W).$$
Therefore,
$$
N_W^b(\delta)\leq N_{\varphi_I(\Lambda_m^b)}(\delta) \leq M_0\mu(\varphi_I(\Lambda_m^b)) \leq  M_02dQ^m\mu(W)\delta^{-\alpha},
$$
where the second inequality holds since  $\mu$ is a box-counting measure.
\end{proof}

Finally, we characterize when $\Lambda$ admits inner trivial points.

Let $\Phi$ be an IFS with attractor $K$, and assume that  $\Phi$
satisfies the strong open set condition with an open set $V$.
A trivial point $x\in K$ is called an \emph{inner trivial point} of $K$ if $x\in V$.
Following Zhang and Huang \cite{ZH22}, a clopen set (closed and open set) $F$ of $\Lambda$  is called a  \emph{island}
 of $\Lambda$ if $F\cap \partial \left( \left[ 0,1 \right] ^d \right) =\emptyset$.
 Obviously, an island $F$ is a union of several $k$-th cylinders for $k$ large enough.

 Zhang and Huang \cite{ZH22}  proved that if $\Phi$ is an IFS
satisfying the strong open set condition, then its attractor   $K$ possesses inner trivial points if and only
if $K$ admits islands.

\begin{lem}\label{lem:island}   Let $\Lambda$ be a non-degenerated  diagonal self-affine sponge of Lalley-Gatzouras type.  Then $\Lambda$ has inner trivial points if and only if  there exist $\delta>0$
and a $\delta$-connected component $X$ of $\Lambda$ such that $X\cap \partial [0,1]^d= \emptyset$.
 \end{lem}

 \begin{proof}
 Clearly $V=(0,1)^d$ is an open set fulfilling the open set condition.
 $\Lambda$ is non-degenerated implies that $\Lambda$ satisfies the strong open set condition
 for this $V$. Hence, by the general result of \cite{ZH22}, $\Lambda$ admits inner trivial points
 if and only if $\Lambda$ admits islands.

 Suppose $\Lambda$ admits an island $U$. Let $d_0$ be the distance between $U$ and $\Lambda\setminus U$, and let $\delta<d_0/3$. Let $X$ be a $\delta$-connected component intersecting $U$.
 Then clearly $X$ does not intersect $\Lambda\setminus U$, and hence does not intersect $\partial [0,1]^d$.

 On the other hand, suppose that there   exists $\delta>0$
and a $\delta$-connected component $X$ of $\Lambda$ such that $X\cap \partial [0,1]^d= \emptyset$.
Let $m$ be an integer large enough so that every $m$-th cylinder has diameter smaller than $\delta$.
Then union of the $m$-th cylinders intersecting $X$ forms an island of $\Lambda$.
 \end{proof}

\section{\textbf{The case that $\Lambda_j$ contains no inner trivial points for some $j$}}


In this section, we always assume that
  $\Lambda$  is a non-degenerated self-affine sponge of Lalley-Gatzouras type.

\begin{thm}\label{thm:no-trivial-point} Suppose $\Lambda$ does not contain inner trivial points. Let $\tau\in(1,+\infty)$.  Then  there exist
 $0<\chi<\dim_B \Lambda$ and $C_1>0$  such that for any cylinder $W$ of $\Lambda$,
\begin{equation}\label{eq:eta}
h_W\left( \delta \right) \leq C_1\mu( W) \delta ^{-\chi}, \quad  \text{ for } \delta\leq (r_*S(W))^\tau.
\end{equation}
Consequently, $\Lambda$ does not satisfy the maximal power law.
\end{thm}

\begin{proof}
Let $W=\Lambda_I$  be a cylinder and let $\delta<S(W)$.
Let $m$ be the  integer such that
\begin{equation}\label{eq:F}
(r_*)^{m+1} \leq \frac{\delta}{S(W)}<(r_*)^{m}.
\end{equation}

Since $\Lambda$ does not possess inner trivial points, by Lemma \ref{lem:island},
every $\delta$-connected component
of $\Lambda$ must intersect $\Lambda\cap \partial [0,1]^d$.
It follows that every $\delta$-connected component of $W$ must intersect
$W\cap \varphi_I(\partial [0,1]^d)$, so
$$
\begin{array}{rlr}
h_W(\delta) &\leq h_{W\cap \varphi_I(\partial [0,1]^d)}(\delta) &\\
    &\leq 3^dN_W^b(\delta) &\text{(By Lemma \ref{lem:box} (i).)}\\
    &\leq  2d3^dM_0Q^m\mu(W)\delta^{-\alpha}. &\text{(By Lemma \ref{lem:boundary-box}.)}
    \end{array}
$$
 By the left side inequality of \eqref{eq:F}, we have that for
\begin{equation}
\delta\leq (r_*S(W))^\tau,
\end{equation}
it holds that
$$
Q^m\leq \delta^s,
$$
where $s=(1-\frac{1}{\tau}) \frac{\log Q}{\log r_*}.$
Hence, the theorem holds by setting $\chi=\alpha-s$.
\end{proof}

From now on,  we fix $\tau_0$  to be the number
$$
\tau_0=\min_{j=1,\dots, d-1}\frac{\log \max\{\phi'_{\ba, j+1};~\ba\in \SD\}}{\log \min\{\phi'_{\ba, j};~\ba\in \SD\}}.
$$
Then $\tau_0>1$ by the coordinate ordering condition; moreover, for any cylinder $W$ of $\Lambda$, we have that
\begin{equation}\label{eq:chip}
S(\pi_{j+1}(W))\leq S(\pi_{j}(W))^{\tau_0}, \quad j=1,\dots,d-1.
\end{equation}

\begin{lem}\label{lem:grow}
  Suppose there exist $0<\chi <\alpha _{d-1}$ and $C_1>0$ such that for every
   cylinder $W'$ of $\Lambda_{d-1}$, it holds that
\begin{equation}\label{eq:chi}
h_{W'}\left( \delta) \leq C_1( \mu _{d-1}\left( W' \right) \delta ^{-\chi} \right) \quad \text{ for } \delta\leq (r_*S(W'))^{\tau_0}.
\end{equation}
Then  for every cylinder $W$ of $\Lambda$, we have
\begin{equation}\label{eq:eta}
h_W\left( \delta \right) \leq   {2^d}{r_*} ^{-(1+\tau_0)d}C_1 \mu _d( W) \delta ^{-\eta} \quad \text{ for } \delta\leq  S(W),
\end{equation}
where $\eta=\chi+\beta_d$.
\end{lem}

 \begin{proof}
 Let $H$ be a cylinder of $\Lambda$ and  denote $H'=\pi_{d-1}(H)$.
 We claim that for $\varepsilon<\delta/2$, it holds that
 \begin{equation}\label{eq:epsilon}
 S(H)< \delta/2 \Rightarrow h_H(\delta) \leq  h_{H'}\left( \varepsilon \right).
\end{equation}

 Pick $x,y\in H$. If $|\pi_{d-1}(x)-\pi_{d-1}(y)|\leq\varepsilon\leq  \delta/2$, then $|x-y|\leq \delta$ since $S(H)\leq \delta/2$.
So if $\pi_{d-1}(x)$ and $\pi_{d-1}(y)$ belong to  a same $\varepsilon$-connected component of $H'$, then $x$ and $y$ belong to  a same $\delta$-connected component of $H$, which proves our claim.

Now we turn to prove the lemma. Pick a cylinder $W$ of $\Lambda$ and let $\delta<S(W)$.
Set
$$\delta'=\delta r_*/2 \text{ and } \varepsilon= {\delta r_*^{1+\tau_0}}/{2}.$$
Then for $H\in \SV_{\delta'}$, we have
$\delta r_*/2\leq S(H)< \delta/2$. Moreover,  by \eqref{eq:chip}, we have
$$\varepsilon\leq r_*^{\tau_0} S(H)\leq (r_*S(H'))^{\tau_0}.$$
Therefore,
$$
\begin{array}{rlr}
h_W(\delta) & \le \sum_{H\in \mathcal{V}_{\delta'}\,\,and\,\,H\subset W}h_{H}(\delta )
& \text{(By Lemma \ref{lem:box}(ii))}\\
&\leq \sum_{H\in \mathcal{V}_{\delta'}\,\,and\,\,H\subset W} h_{H'}(\varepsilon)
 &\text{(By \eqref{eq:epsilon}  and $S(H)<\delta/2$)}\\
&\leq \sum_{H\in \mathcal{V}_{\delta'}\,\,and\,\,H\subset W}
{C_1 \mu _{d-1}(H') \varepsilon ^{-\chi}} & \text{(By \eqref{eq:chi})}\\
&=\sum_{H\in \mathcal{V}_{\delta'}\,\,and\,\,H\subset W}
 C_1  \mu _d(H)  S(H)^{-\beta_d}  \varepsilon ^{\beta_d-\eta}  &\\
&\leq 2^\eta r_*^{\tau_0 \beta_d-(1+\tau_0)\eta} C_1  \mu _d ( W ) \delta ^{-\eta}
&\text{(Since $(S(H)/\varepsilon)\geq r_*^{\tau_0}$)}\\
&\leq 2^d r_*^{-(1+\tau_0)d} C_1  \mu _d ( W ) \delta ^{-\eta}, & \text{(Since $\eta< d$ and $r_*<1$)}
\end{array}
$$
which confirms \eqref{eq:eta}. The lemma is proved.
\end{proof}

\begin{coro}\label{cor:good} Let $\Lambda$ be a non-degenerated self-affine sponge of Lalley-Gatzouras type. If there exists
$j\in\{1,\dots, d\}$ such that $\Lambda_j$ does not contain inner trivial points, then there exists
  a real number $0<\chi <\dim_B\Lambda$ such that for any cylinder $R$,
\begin{equation}\label{eq:dropp}
h_R\left( \delta \right) =O\left( \mu \left( R \right) \delta ^{-\chi} \right) ,\ \ \delta \rightarrow 0.
\end{equation}
\end{coro}

\begin{proof} If $j=d$, then the corollary holds by Theorem \ref{thm:no-trivial-point}.
If $j<d$, then by Lemma \ref{lem:grow} and induction, we  obtain \eqref{eq:dropp}.
\end{proof}

\section{\textbf{Proof of Theorem \ref{thm:affine}}}

 In this section, we always assume that $\Lambda$ is a non-degenerated   self-affine sponge
of Lalley-Gatzouras type, and denote  $\mu_j$   the canonical Bernoulli measure of $\Lambda_j$.
Especially,   $\mu:=\mu_d$.

 Let $W=\Lambda_I$ be a $k$-th cylinder of $\Lambda$. We call $\phi_I \left( \left[ 0,1 \right] ^d \right)$  the corresponding \emph{$k$-th basic pillar of $W$}. A $\delta$-connected component $U$ of $W$ is called an
\emph{inner $\delta$-connected component}, if
$$
  U\subset \phi_I(0,1)^d;
$$
otherwise, we call $U$ a \emph{boundary $\delta$-connected component}. We denote by $h_{W}^{b}\left( \delta \right)$ the number of boundary $\delta$-connected components of $W$, and by $h_{W}^{i}\left( \delta \right)$ the number of inner $\delta$-connected components of $W$.

\begin{lem}\label{lem:boundary}    Then for any $\kappa >0$, there exists an integer $m_0\geq 1$
such that for any cylinder $W$ of $\Lambda$ and  any $\delta< (r_*)^{m_0} S(W)$, it holds that
$$
h_{W}^{b}\left( \delta \right) \le \kappa \mu \left( W \right) \delta ^{-dim_B\Lambda}.
$$
\end{lem}

\begin{proof} Denote $\alpha=\dim_B \Lambda$.
  By Proposition \ref{prop:sponge}, there exists $M_0>0$ such that for any cylinder $H$ of $\Lambda$,
$$
M_0^{-1}\mu(H) \delta ^{-\alpha}\le N_H (\delta) \le M_0\mu(H) \delta ^{-\alpha} \quad \text{ for } \delta <S(H).
$$

 Let $W=\Lambda_I$ be a cylinder of $\Lambda$ and let $m\geq 1$.
 Let ${\mathcal B}(W, m)$ be the collection of
      $(k+m)$-th cylinders  of $W$   whose
    corresponding basic pillars intersect  $ \varphi_I(\partial \left[ 0,1 \right] ^d) $.
 Then
     $$\cup {\mathcal B}(W, m) =\varphi_I(\Lambda_{m}^b):=W^*,$$
where $\Lambda_{m}^b$ is defined by \eqref{eq:mb}.
 By Lemma \ref{lem:QQ}, we have
$$
 \mu(W^*)=\mu(\Lambda_{m}^b)\mu(W) \leq 2dQ^m\mu(W).
 $$

Let $m_0$ be an integer such that $3^dM_0(2d)Q^{m_0} <\kappa $.
Let $\delta<  (r_*)^{m_0} S(W)$; in this case
 if $U$ is a boundary $\delta$-connected component of $W$, then $U$
 must intersect  $W^*$. Therefore,
$$
\begin{array}{rlr}
h_{W}^{b} ( \delta  )  &\leq  h_{W^*}(\delta)
\leq \sum_{H\in {\mathcal B}(W, m_0)} h_H(\delta) &\\
&\leq 3^d\sum_{H\in {\mathcal B}(W, m_0)} N_H(\delta)  &\text{(By Lemma \ref{lem:box} (i).)}\\
& \le 3^d\sum_{H\in {\mathcal B}(W, m_0)} M_0\mu(H)\delta^{-\alpha}    &\\
&\leq 3^d M_0 \mu  ( W^* ) \delta ^{-\alpha} &\\
&\leq \kappa \mu(W) \delta ^{-\alpha},
\end{array}
$$
which proves the lemma.
\end{proof}



\begin{thm}\label{thm:grow-2} If $\mu_{d-1}$  is a component-counting measure of $\Lambda_{d-1}$
and $\Lambda$ possesses inner trivial points. Then $\mu$ is a component-counting measure of $\Lambda$.
\end{thm}

\begin{proof} Since $\Lambda$ is non-degenerated and possesses trivial points, by Lemma \ref{lem:island},
  there is a clopen subset $F$ of $\Lambda$ such that $F\cap \partial [0,1]^d=\emptyset$.
  Assume that  $F$  is the union  of $L$ number of  $m_1$-th cylinders.
  Let $U$ be a $m_1$-cylinder in $F$ with the maximal measure in $\mu$, then
$$\mu(U)\geq \mu(F)/L.$$

Let $W$ be a $k$-th cylinder of $\Lambda$ and denote $W'=\pi _{d-1}\left( W \right)$. Let $\delta\leq S(W)$.
Since
$$
h_W\left( \delta \right) \le 3^d N_W(\delta)\asymp \mu _d\left( W \right) \delta ^{-\alpha _d},
$$
we need only to consider the lower bound estimate of $h_W(\delta)$.

That $\mu_{d-1}$  is a component-counting measure implies that
there is a constant $M$ such that
$$
h_{W'}\left( \delta \right) \geq M \mu _{d-1}\left( W' \right) \delta ^{-\alpha _{d-1}},\ \quad \text{ for all } \delta <S(W').
$$
(We remark that we set the threshold $\delta_0(W')=S(W')$ here.)

Let $\kappa=M/2$ and let $m_0$ be the constant in Lemma \ref{lem:boundary}. Set
$\delta'= \delta/(r_*)^{m_0+m_1}$. Let $\mathcal{V}_{\delta'}$  be the $\delta'$-blocking of $\Lambda$.

Pick $H\in \mathcal{V}_{\delta'}$ and $H\subset W$.
Let $f_H$ be the affine map such that  $H=f_H(\Lambda)$, then $H_F:=f_H(F)$ is an island of $H$ in the sense that
$f_H(F)$ is a clopen subset of $H$, and $f_H(F)$ does not intersect $f_H(\partial [0,1]^d)$.
Clearly
$$
{\mu(H_F)}= \mu(F){\mu(H)}.
$$

 Now we estimate $h_{H_F}(\delta)$.  Denote
 $$U_H=f_H(U) \text{ and } U_H'=\pi_{d-1}(U_H).$$
 First, since $\pi_{d-1}$ is contractive, by Lemma \ref{lem:box}(iii),  we have
 $$
 h_{U_H}(\delta)\geq h_{U_H'}(\delta).
 $$
Notice that
$$\delta = (r_*)^{m_0+m_1} \delta'\leq (r_*)^{m_0+m_1} S(H)\leq   (r_*)^{m_0} S(U_H)
\leq   (r_*)^{m_0} S(U_H').
$$
On one hand, since $\mu_{d-1}$  is a component-counting measure, we have
\begin{equation}\label{eq:Uprime}
h_{U_H'}(\delta)   \geq  M \mu_{d-1}(U_H') \delta ^{-\alpha_{d-1}}.
\end{equation}
On the other hand, by Lemma \ref{lem:boundary}, we have
$$
h_{U_H'}^{b}( \delta) \le \kappa \mu_{d-1}  ( U_H'  ) \delta ^{-\alpha_{d-1}}=\frac{M}{2} \mu_{d-1}  ( U_H'  ) \delta ^{-\alpha_{d-1}}.
$$
This together with \eqref{eq:Uprime} imply that
\begin{equation}\label{eq:key}
h_{U_H'}^{i}(\delta) \ge  \frac{M}{2}\mu_{d-1}(U_H') \delta ^{-\alpha_{d-1}}.
\end{equation}
First, we show that $h_{H_F}(\delta)$ is no less than the left hand side of \eqref{eq:key}:
\begin{equation}\label{eq:left}
h_{H_F}(\delta)\geq h_{\pi_{d-1}(H_F)}(\delta)\geq h_{U_H'}^{i}(\delta).
\end{equation}
Secondly, we show that $\mu(H)$ is comparable with the right hand side of \eqref{eq:key}.
Since
\begin{equation}\label{eq:tree}
\delta = (r_*)^{m_0+m_1} \delta'\geq (r_*)^{m_0+m_1+1} S(H)\geq (r_*)^{m_0+m_1+1}  S(U_H),
\end{equation}
we have
\begin{equation}\label{eq:right}
\begin{array}{rl}
 \mu_{d-1}(U_H') \delta ^{-\alpha_{d-1}} &=\mu(U_H)\left (\frac{\delta}{S(U_H)}\right )^{\beta_{d}}\delta ^{-\alpha_{d}} \\
&\geq \mu(U_H) r_*^{(m_0+m_1+1)}\delta ^{-\alpha_{d}} \quad \text{(By \eqref{eq:tree} and using the fact $\beta_{d}\leq 1$.)}\\
&\geq  r_*^{(m_0+m_1+1)}\frac{ \mu(H_F)}{L} \delta ^{-\alpha_{d}}\\
&=  r_*^{(m_0+m_1+1)}\frac{\mu(F) }{L} \mu(H)\delta ^{-\alpha_{d}}.
\end{array}
\end{equation}
Therefore,
$$
\begin{array}{rlr}
h_W ( \delta  ) & \ge \sum_{H\in \SV_{\delta'} \text{ and } H\subset W} h_{H_F}(\delta)
&\text{(Since $H_F$ are `islands'.)}\\
& \geq \sum_{H\in \SV_{\delta'} \text{ and } H\subset W} h_{U_H'}^{i}(\delta)&\text{(By equation \eqref{eq:left}.)}\\
& \geq \sum_{H\in \SV_{\delta'} \text{ and } H\subset W}  \frac{M}{2}\mu_{d-1}(U_H') \delta ^{-\alpha_{d-1}} &\text{(By equation \eqref{eq:key}.)}\\
& \geq \sum_{H\in \SV_{\delta'} \text{ and } H\subset W} \frac{Mr_*^{(m_0+m_1+1)}\mu(F)}{2L}\mu\left( H \right) \delta ^{-\alpha} &\text{(By equation \eqref{eq:right}.)}\\
& =  \frac{Mr_*^{(m_0+m_1+1)}\mu(F)}{2L}\mu\left( W \right) \delta ^{-\alpha}.
\end{array}
$$
 This proves that $\mu$ is a component-counting measure and
 the threshold $\delta_0(W)$ can be set to be $S(W)$.
\end{proof}

\begin{proof}[\textbf{Proof of Theorem  \ref{thm:affine}}]  The second assertion in the theorem is proved in Corollary \ref{cor:good}. In the following, we show that
 (i)$\Rightarrow$(ii) $\Rightarrow$ (iii)$ \Rightarrow$ (i).

(i)$\Rightarrow$(ii) is trivial, see Remark \ref{rem:homo}.

(ii)$\Rightarrow$(iii) holds by  Corollary  \ref{cor:good}.

(iii)$\Rightarrow$(i): First, $\mu_1$  is a component-counting measure of $\Lambda_1=\pi_1(\Lambda)$ by Theorem \ref{thm:conformal}.
Then by Theorem \ref{thm:grow-2} and induction, we conclude that
    $\mu_j$ is a component-counting measure of $\Lambda_j$
    for each $1\le j\le d$.

    That (iii)$\Leftrightarrow (iv)$ is obvious. The theorem is proved.
\end{proof}

\begin{proof}[\textbf{Proof of Theorem  \ref{thm:dropp}}]
By Theorem \ref{thm:affine}, if $\Lambda$ does not satisfy the maximal power law,
then there exists $1\leq j\leq d$ such that $\Lambda_j$ contains no inner trivial points.
Now the theorem is  a direct consequence of Corollary  \ref{cor:good}.
\end{proof}


\begin{thebibliography}{99}
\addcontentsline{toc}{chapter}{Bibliography}

\bibitem{Baranski07}
K. Bara\'nski, \emph{Hausdorff dimension of the limit sets of some planar geometric constructions}, Advances in Mathematics, 2007, \textbf{210}(1): 215-245.

\bibitem{BK21} A. Banaji and I. Kolossvary, Intermidiate dimensions of Bedford-McMullen carpets with applications to Lipschitz equivalence. 2021.

\bibitem{Bar07} J. Barral and M. Mensi, \emph{Gibbs measures on self-affine Sierpi\'nski carpets and their singularity spectrum}, Ergod. Th. \& Dynam. Sys., 2007, \textbf{27}(5): 1419-1443.

\bibitem{Bed84}
T. Bedford, \emph{Crinkly curves, Markov partitions and dimensions}, phD Thesis, University of Warwick, 1984.

\bibitem{BT54}
A.S. Besicovitch and S.J. Taylor, \emph{On the complementary intervals of a linear closed set of zero Lebesgue measure}, J. London Math. Soc. \textbf{29} (1954), 449-459.

\bibitem{Das16}
T. Das, D. Simmons, \emph{The Hausdorff and dynamical dimensions of self-affine sponges: a dimension gap result}, Inventiones Mathematicae, 2016, \textbf{2}: 1-50.



\bibitem{DWX15}
J. Deng, Q. Wang, L.F. Xi, \emph{Gap sequences of self-conformal sets}, Arch. Math., \textbf{104} (2015), 391-400.



 \bibitem{Falconer90}
K.J. Falconer, \emph{Fractal geometry: mathematical foundations and applications}, John Wiley \& Sons, 1990.

\bibitem{Fal95}
K.J. Falconer, \emph{On the Minkowski measurability of fractals}, Proc. Amer. Math. Soc, 1995, \textbf{123}: 1115-1124.




\bibitem{FW05}
D.J. Feng,  Y. Wang, \emph{A Class of Self-Affine Sets and Self-Affine Measures}, Journal of Fourier Analysis \& Applications, 2005, \textbf{11}(1): 107-124.


\bibitem{HR21}
L.Y. Huang, H. Rao, \emph{A dimension drop phenomenon of fractal cubes}, J. Math. Anal. Appl., \textbf{497} (2021), 11pp.

\bibitem{HRWX} L.Y. Huang, H. Rao, Z.Y. Wen and Y.L. Xu, Box-counting measure of metric spaces and its applications, Preprint 2022.


\bibitem{HZ22}
L.Y. Huang, Y. Zhang, A dichotomy of strong open set condition of self-conformal sets, Fractals, \textbf{30}(10) (2022), 6 pp.


\bibitem{KP96}
R. Kenyon, Y. Peres, \emph{Measures of full dimension on affine-invariant sets}, Egodic Theory and Dynamical Systems, 1996, \textbf{16}: 307-323.

\bibitem{King95}
J.F. King, \emph{The Singularity Spectrum for General Sierpi\'nski Carpets}. Advances in Mathematics, 1995, \textbf{116}(1): 1-11.

\bibitem{Lalley92}
S.P. Lalley, D. Gatzouras, \emph{Hausdorff and box dimensions of certain self-affine fractals}, Indiana Univ. Math. J., 1992, \textbf{41}(2): 533-568.


\bibitem{LP93}
M.L. Lapidus, C. Pomerance, \emph{The Riemann zeta-function and the one-dimensional Weyl-Berry conjecture for fractal drums}, Proc. Lond. Math. Soc, \textbf{66} (1993), 41-69.

\bibitem{LRY99}
K.S. Lau, H. Rao, Y.L. Ye, \emph{Corrigendum: ``Iterated function system and Ruelle operator''}, J. Math. Anal. Appl., \textbf{262} (2001), 446-451.

\bibitem{LLM13}
B.M. Li, W.X. Li, J.J. Miao, \emph{Lipschitz equivalence of McMullen sets}, Fractals, 2013, \textbf{21}(3 \& 4), 1350022, 11 pages.

\bibitem{LMR20} Z. Liang, J.J. Miao, H. J. Ruan, \emph{Gap sequences and Topological properties of Bedford-McMullen sets},
Preprint (2021), (arXiv:2003.08262 [math.PH]).

\bibitem{LR18}
Z. Liang, H. J. Ruan, \emph{Gap sequences of fractal squares}, J. Math. Anal. Appl. \textbf{472} (2019) 1475-1486.







\bibitem{JR11}
T. Jordan, M. Rams, \emph{Multifractal analysis for Bedford-McMullen carpets}, Math. Proc. Cambridge Philos. Soc., 2011, \textbf{150}: 147-156.

\bibitem{MM11}
J.M. Mackay, \emph{Assouad dimension of self-affine carpets}. (English summary) Conformal Geometry \& Dynamics, 2011, \textbf{15}(12): 177-187.

\bibitem{Mc84}
C. McMullen, \emph{The Hausdorff dimension of Sierpi$\acute{\text{n}}$ski carpets}. Nagoya Math. J., 1984, \textbf{966}: 1-9.



\bibitem{MXX17}
J.J. Miao, L.F. Xi, Y. Xiong, \emph{Gap sequences of McMullen sets}, Proc. Amer. Math. Soc., \textbf{145} (2017), 1629-1637.


\bibitem{Olsen07}
L. Olsen, \emph{Symbolic and geometrical local dimensions of self-affine multifractal Sierpi$\acute{\text{n}}$ski sponges in $\mathbb{R}^d$}, Stochastics and Dynamics, 2007, \textbf{7}(01): 37-51.

\bibitem{Patzschke97}
N. Patzschke, \emph{Self-conformal multifractal measures}, Adv. Appl. Math., \textbf{19} (1997), 486-513.

\bibitem{Peres01}
Y. Peres, M. Rams, K. Simon, B. Solomyak, \emph{Equivalence of positive Hausdorff measure and the open set condition for self-conformal sets}, Proc. Amer. Math. Soc., \textbf{129}(2001), 2689-2699.

\bibitem{Peres94}
Y. Peres, \emph{The self-affine carpets of McMullen and Bedford have infinite Hausdorff measure}, Math. Proc. Cambridge
Philos. Soc., 1994, \textbf{116}: 513-526.


\bibitem{RRY08}
H. Rao, H.J. Ruan, Y.M. Yang, \emph{Gap sequence, Lipschitz equivalence and box dimension of fractal sets}, Nonlinearity, 2008, \textbf{21}: 1339-1347.

\bibitem{Rao19}
H. Rao, Y.M. Yang, Y. Zhang, \emph{The bi-Lipschitz
classification  of Bedford-McMullen carpets (I): Invariance of multifractal spectrum and measure preserving property}, 2019, Preprint (arXiv:2005.07451 [math.DS]).

\bibitem{Reeve10}
H.W. Reeve, \emph{Multifractal analysis for Birkhoff averages on Lalley-Gatzouras repellers}, Fundamenta Mathematicae, 2010, \textbf{212} (1).


\bibitem{YZ20}
Y.M. Yang, Y. Zhang, \emph{Lipschitz classification of  Bedford-McMullen carpets with uniform horizontal fibers}, J. Math.  Anal.  Appl., 2020, \textbf{495}(2): 124742.


\bibitem{ZH22}
Y. Zhang, L.Y. Huang, \emph{Relations between topological and metrical properties of self-affine Sierpi$\acute{\text{n}}$ski sponges}, J. Math. Anal. Appl., \textbf{514} (2022), https://doi.org/10.1016/j.jmaa.2022.126295.


\bibitem{ZX22}
Y.F. Zhang, Y.L. Xu, \emph{Dimension drop of connected part of slicing self-affine sponges}, J. Math. Anal. Appl., 2022, \textbf{509}, 13 pp.
\end{thebibliography}
\end{document}